\documentclass[platex]{amsart}
\usepackage{amssymb,amsmath,amsthm,empheq,comment,color,url}

\usepackage[top=30truemm,bottom=30truemm,left=20truemm,right=20truemm]{geometry}

\theoremstyle{definition}
\newtheorem{theorem}{Theorem}[section]
\newtheorem{proposition}[theorem]{Proposition}

\newtheorem{lemma}[theorem]{Lemma}
\newtheorem{corollary}[theorem]{Corollary}

\newtheorem*{remark*}{Remark}

\DeclareMathOperator{\Mod}{Mod}

\title{Remarks on minimal mass blow up solutions for a double power nonlinear Schr\"{o}dinger equation}
\author[N. Matsui]{Naoki Matsui}
\date{\today}
\address[N. Mastui]{Department of Mathematics\\ Tokyo University of Science\\ 1-3 Kagurazaka, Shinjuku-ku, Tokyo 162-8601, Japan}
\email[N. Matsui]{1120703@ed.tus.ac.jp}

\begin{document}
\begin{abstract}
We consider the following nonlinear Schr\"{o}dinger equation with double power nonlinearity
\[
i\frac{\partial u}{\partial t}+\Delta u+|u|^{\frac{4}{N}}u+|u|^{p-1}u=0,\quad 1<p<1+\frac{4}{N}
\]
in $\mathbb{R}^N$. For $N=1,2,3$, Le Coz-Martel-Rapha\"{e}l (2016) construct a minimal-mass blow-up solution. Moreover, the previous study derives blow-up rate of the blow-up solution. In this paper, we extend this result to the general dimension. Furthermore, we investigate the behaviour of the critical mass blow-up solution near the blow-up time.
\end{abstract}

\maketitle

\section{Introduction}
We consider the following nonlinear Schr\"{o}dinger equation with double power nonlinearity
\begin{empheq}[left={(\mathrm{NLS\pm})\ \empheqlbrace\ }]{align*}
&i\frac{\partial u}{\partial t}+\Delta u+|u|^{\frac{4}{N}}u\pm|u|^{p-1}u=0,\\
&u(t_0)=u_0\nonumber
\end{empheq} 
in $\mathbb{R}^N$, where
\[
1<p<1+\frac{4}{N}
\]
Then, (NLS$\pm$) is locally well-posed in $H^1(\mathbb{R}^N)$ (\cite{CSSE}). This means that for any $u_0\in H^1(\mathbb{R}^N)$, there exists a unique maximal solution $u\in C^1((-T_*,T^*),H^{-1}(\mathbb{R}^N))\cap C((-T_*,T^*),H^1(\mathbb{R}^N))$. Moreover, the mass (i.e., $L^2$-norm) and energy $E$ of the solution are conserved by the flow, where
\begin{eqnarray}
\label{energy}
E(u):=\frac{1}{2}\left\|\nabla u\right\|_2^2-\frac{1}{2+\frac{4}{N}}\left\|u\right\|_{2+\frac{4}{N}}^{2+\frac{4}{N}}\mp\frac{1}{p+1}\|u\|_{p+1}^{p+1}.
\end{eqnarray}
Furthermore, there is a blow-up alternative
\[
T^*<\infty\ \Rightarrow\ \lim_{t\nearrow T^*}\left\|\nabla u(t)\right\|_2^2=\infty.
\]

\subsection{Main results}
In this paper, for (NLS$\pm$), we prove the following result, which is stronger than Le Coz, Martel, and Rapha\"{e}l \cite{LMR}.

\begin{theorem}[Existence of a minimal mass blow-up solution]
\label{theorem:exist}
For any energy level $E_0\in\mathbb{R}$, there exist $t_0<0$ and a critical mass radial initial value $u(t_0)\in \Sigma^1(\mathbb{R}^N)$ with $E(u_0)=E_0$ such that the corresponding solution $u$ for (NLS$+$) blows up at $T^*=0$. Moreover,
\[
\left\|u(t)-\frac{1}{\lambda(t)^\frac{N}{2}}P\left(t,\frac{x}{\lambda(t)}\right)e^{-i\frac{b(t)}{4}\frac{|x|^2}{\lambda(t)^2}+i\gamma(t)}\right\|_{\Sigma^1}\rightarrow 0\quad (t\nearrow 0)
\]
holds for some blow-up profile $P$, positive constants $C_1(p)$ and $C_2(p)$, positive-valued $C^1$ function $\lambda$, and real-valued $C^1$ functions $b$ and $\gamma$ such that
\begin{align*}
&P(t)\rightarrow Q\mbox{ in }H^1(\mathbb{R}^N),\quad\lambda(t)=C_1(p)|t|^{\frac{4}{4+N(p-1)}}\left(1+o(1)\right),\\
&b(t)=C_2(p)|t|^{\frac{4-N(p-1)}{4+N(p-1)}}\left(1+o(1)\right),\quad \gamma(t)^{-1}=O\left(|t|^{\frac{4-N(p-1)}{4+N(p-1)}}\right)
\end{align*}
as $t\nearrow 0$.
\end{theorem}

Here, $\Sigma^1$ is defined as
\[
\Sigma^1:=\left\{u\in H^1\left(\mathbb{R}^N\right)\ \middle|\ xu\in L^2\left(\mathbb{R}^N\right)\right\}.
\]

\begin{theorem}[Non-existence of a minimal mass blow-up solution (\cite{LMR})]
\label{theorem:nonexist}
For any critical-mass initial value $u(t_0)\in H^1(\mathbb{R}^N)$, the corresponding solution for (NLS$-$) is global and bounded in $H^1(\mathbb{R}^N)$.
\end{theorem}

\begin{theorem}[Existence of a supercritical mass blow-up solution (\cite{LMR})]
\label{theorem:existcm}
For any $\delta$, there is $u(t_0)\in H^1(\mathbb{R}^N)$ such that $\|u_0\|_2=\|Q\|_2+\delta$ and the corresponding solution for (NLS$-$) blows up at finite time.
\end{theorem}

Proofs of Theorem \ref{theorem:nonexist} and Theorem \ref{theorem:existcm} in \cite{LMR} is dimension-independent. In this paper, we prove only Theorem \ref{theorem:exist}.

\subsection{Notations}
In this section, we introduce the notation used in this paper.

Let
\[
\mathbb{N}:=\mathbb{Z}_{\geq 1},\quad\mathbb{N}_0:=\mathbb{Z}_{\geq 0}.
\]
Unless otherwise noted, we define
\begin{align*}
&(u,v)_2:=\Re\int_{\mathbb{R}^N}u(x)\overline{v}(x)dx,\quad \left\|u\right\|_q:=\left(\int_{\mathbb{R}^N}|u(x)|^qdx\right)^{\frac{1}{q}},\\
&f(z):=|z|^{\frac{4}{N}}z,\quad  F(z):=\frac{1}{2+\frac{4}{N}}|z|^{2+\frac{4}{N}},\quad g(z):=|z|^{p-1}z,\quad  G(z):=\frac{1}{p+1}|z|^{p+1}.
\end{align*}
By identifying $\mathbb{C}$ with $\mathbb{R}^2$, we denote the differentials of the functions $df$, $dg$, $dF$, and $dG$. We define
\[
\Lambda:=\frac{N}{2}+x\cdot\nabla,\quad L_+:=-\Delta+1-\left(1+\frac{4}{N}\right)Q^{\frac{4}{N}},\quad L_-:=-\Delta+1-Q^{\frac{4}{N}}.
\]
Then,
\[
L_-Q=0,\quad L_+\left(\Lambda Q\right)=-2Q,\quad L_-\left(|x|^2Q\right)=-4\Lambda Q,\quad L_+\rho=|x|^2 Q
\]
holds, where $\rho$ is the unique radial Schwartz solution of $L_+\rho=|x|^2 Q$. Furthermore, there is a $\mu>0$ such that
\[
\forall u\in H_{\mathrm{rad}}^1(\mathbb{R}^N),\quad \left( L_+\Re u,\Re u\right)_2+\left( L_-\Im u,\Im u\right)_2\geq \mu\left\|u\right\|_{H^1}^2-\frac{1}{\mu}\left((\Re u,Q)_2^2+(\Re u,|x|^2 Q)_2^2+(\Im u,\rho)_2^2\right)
\]
(e.g., see \cite{MRO,MRUPB,RSEU,WL}). We introduce
\[
\Sigma^m:=\left\{u\in H^m(\mathbb{R}^N)\ \big|\ |x|^m u\in L^2(\mathbb{R}^N)\right\}.
\]
and denote by $\mathcal{Y}$ the set of functions $h\in C^{\infty}_{\mathrm{rad}}(\mathbb{R}^N)$ such that
\[
\forall\alpha\in{\mathbb{N}_0}^N\exists C_\alpha,\kappa_\alpha>0,\ \left|\left(\frac{\partial}{\partial x}\right)^\alpha h(x)\right|\leq C_\alpha(1+|x|)^{\kappa_{\alpha}}Q(x).
\]
Finally, we use $\lesssim$ and $\gtrsim$ when the inequalities hold except for non-essential positive constant differences and $\approx$ when $\lesssim$ and $\gtrsim$ hold.

\section{Preliminaries}
\label{sec:Preliminaries}
We provide the following statements regarding notations.

\begin{proposition}
\label{GSP}
For any $\alpha\in{\mathbb{N}_0}^N$, there is a constant $C_\alpha>0$ such that $\left|\left(\frac{\partial}{\partial x}\right)^\alpha Q(x)\right|\leq C_\alpha Q(x)$. Similarly, $\left|\left(\frac{\partial}{\partial x}\right)^\alpha \rho(x)\right|\leq C_\alpha(1+|x|)^{\kappa_\alpha} Q(x)$ holds (e.g., \cite{LMR,Ninv}).
\end{proposition}

\begin{lemma}
For the ground state $Q$, 
\[
(Q,\rho)_2=\frac{1}{2}\bigl\||x|^2Q\bigr\|_2^2
\]
holds.
\end{lemma}

\begin{lemma}
For an appropriate function $w$,
\[
\left(|x|^{2p}w,\Lambda w\right)_2=-p\bigl\||x|^pw\bigr\|_2^2,\quad (-\Delta w,\Lambda w)_2=\bigl\|\nabla w\bigr\|_2^2,\quad (|w|^qw,\Lambda w)_2=\frac{Nq}{2(q+2)}\|w\|_{q+2}^{q+2}
\]
holds.
\end{lemma}

\begin{lemma}[Properties of $F$ and $f$]
\label{Fdef}
For $F$ and $f$,
\begin{align*}
&\frac{\partial F}{\partial\Re}=\Re f,\quad \frac{\partial F}{\partial\Im}=\Im f,\quad \frac{\partial \Re f}{\partial\Im}=\frac{\partial \Im f}{\partial\Re},\\
&\frac{\partial}{\partial s}F(z(s))=f(z(s))\cdot\frac{\partial z}{\partial s}=\Re\left(f(z(s))\overline{\frac{\partial z}{\partial s}}\right),\\
&dF(z)(w)=f(z)\cdot w=\Re\left(f(z)\overline{w}\right),\\
&df(z)(w_1)\cdot w_2=df(z)(w_2)\cdot w_1,\\
&\frac{\partial}{\partial s}dF(z(s))(w(s))=df(z(s))(w(s))\cdot\frac{\partial z}{\partial s}+f(z(s))\cdot\frac{\partial w}{\partial s},\\
&\frac{\partial}{\partial w}\int_{\mathbb{R}^N}\left(F(z(x)+w(x))-F(z(x))-dF(z(x))(w(x))\right)dx=f(z+w)-f(z),\\
&L_+\left(\Re Z\right)+iL_-\left(\Im Z\right)=-\Delta Z+Z-df(Q)(Z)
\end{align*}
holds. When identifying $\mathbb{C}$ with $\mathbb{R}^2$, $\cdot$ is the inner product of $\mathbb{R}^2$.
\end{lemma}

\section{Construction of a blow-up profile}
In this section, we construct a blow-up profile $P$.

For $K\in\mathbb{N}_0$, let
\[
\Sigma_K=\left\{\ (j,k)\in{\mathbb{N}_0}^2\ \middle|\ j+k\leq K\ \right\}.
\]

\begin{proposition}
Let $K\in\mathbb{N}$ be sufficiently large. Let $\lambda(s)>0$ and $b(s)\in\mathbb{R}$ be $C^1$ function of $s$ such that $\lambda(s)+|b(s)|\ll 1$.

\mbox{(i) Existence of blow-up profile.} For any $(j,k)\in\Sigma_K$, there exist real-valued $P_{j,k}^+,P_{j,k}^-\in\mathcal{Y}$ and $\beta_{j,k}\in\mathbb{R}$ such that $P$ satisfies
\[
i\frac{\partial P}{\partial s}+\Delta P-P+f(P)+\lambda^\alpha g(P)+\theta\frac{|y|^2}{4}P=\Psi,
\]
where $\alpha=2-\frac{N(p-1)}{2}$, and $P$ and $\theta$ are defined by
\begin{align*}
P(s,y)&:=Q(y)+\sum_{(j,k)\in\Sigma_K}\left(b(s)^{2j}\lambda(s)^{(k+1)\alpha}P_{j,k}^+(y)+ib(s)^{2j+1}\lambda(s)^{(k+1)\alpha}P_{j,k}^-(y)\right),\\
\theta(s)&:=\sum_{(j,k)\in\Sigma_K}b(s)^{2j}\lambda(s)^{(k+1)\alpha}\beta_{j,k}.
\end{align*}
Moreover, for some $\epsilon'>0$ that is sufficiently small,
\[
\left\|e^{\epsilon'|y|}\Psi\right\|_{H^1}\lesssim\lambda^\alpha\left(\left|b+\frac{1}{\lambda}\frac{\partial \lambda}{\partial s}\right|+\left|\frac{\partial b}{\partial s}+b^2-\theta\right|\right)+(b^2+\lambda^\alpha)^{K+2}.
\]

\mbox{(ii) Mass and energy properties of blow-up profile.} Let define
\[
P_{\lambda,b,\gamma}(s,x)=\frac{1}{\lambda(s)^\frac{N}{2}}P\left(s,\frac{x}{\lambda(s)}\right)e^{-i\frac{b(s)}{4}\frac{|x|^2}{\lambda(s)^2}+i\gamma(s)}.
\]
Then,
\begin{align*}
\left|\frac{d}{ds}\|P_{\lambda,b,\gamma}\|_2^2\right|&\lesssim\lambda^\alpha\left(\left|b+\frac{1}{\lambda}\frac{\partial \lambda}{\partial s}\right|+\left|\frac{\partial b}{\partial s}+b^2-\theta\right|\right)+(b^2+\lambda^\alpha)^{K+2},\\
\left|\frac{d}{ds}E(P_{\lambda,b,\gamma})\right|&\lesssim\frac{1}{\lambda^2}\left(\left|b+\frac{1}{\lambda}\frac{\partial \lambda}{\partial s}\right|+\left|\frac{\partial b}{\partial s}+b^2-\theta\right|+(b^2+\lambda^\alpha)^{K+2}\right).
\end{align*}
hold. Moreover,
\begin{align}
\left|8E(P_{\lambda,b,\gamma})-\||y|Q\|_2^2\left(\frac{b^2}{\lambda^2}-\frac{2\beta}{2-\alpha}\lambda^{\alpha-2}\right)\right|\lesssim\frac{\lambda^\alpha(b^2+\lambda^\alpha)}{\lambda^2},
\end{align}
where
\[
\beta:=\beta_{0,0}=\frac{2N(p-1)}{p+1}\frac{\|Q\|_{p+1}^{p+1}}{\||y|Q\|_2^2}.
\]
\end{proposition}

\begin{proof}
See \cite{LMR,Ninv} for the proofs. The proofs are dimension-independent.
\end{proof}

\begin{lemma}[Decomposition]
\label{decomposition}
There exist constants $\overline{l},\overline{\lambda},\overline{b},\overline{\gamma}>0$ such that the following logic holds.

Let $I$ be an interval, let $\delta>0$ be sufficiently small, and let $u\in C(I,H^1(\mathbb{R}^N))\cap C^1(I,H^{-1}(\mathbb{R}^N))$ satisfy that there exist functions $\lambda\in \mathrm{Map}(I,(0,\overline{l}))$ and $\gamma\in \mathrm{Map}(I,\mathbb{R})$ such that 
\[
\forall\ t\in I,\ \left\|\lambda(t)^{N/2}u(t,\lambda(t) \cdot)e^{i\gamma(t)}-Q\right\|_{H^1}< \delta.
\]
Then, (given $\tilde{\gamma}(0)$) there exist unique functions $\tilde{\lambda}\in C^1(I,(0,\infty))$ and $\tilde{b},\tilde{\gamma}\in C^1(I,\mathbb{R})$ that are independent of $\lambda$ and $\gamma$ such that 
\begin{align}
\label{mod}
u(t,x)&=\frac{1}{\tilde{\lambda}(t)^{N/2}}\left(P+\tilde{\varepsilon}\right)\left(t,\frac{x}{\tilde{\lambda}(t)}\right)e^{-i\tilde{b(t)}|x|^2/4\tilde{\lambda}(t)^2+i\tilde{\gamma}(t)},\\
\tilde{\lambda}(t)&\in\left(\lambda(t)(1-\overline{\lambda}),\lambda(t)(1+\overline{\lambda})\right),\nonumber\\
\tilde{b}(t)&\in(-\overline{b},\overline{b}),\nonumber\\
\tilde{\gamma}(t)&\in\bigcup_{m\in\mathbb{Z}}(-\overline{\gamma}-\gamma(t)+2m\pi,\overline{\gamma}-\gamma(t)+2m\pi)\nonumber
\end{align}
holds and $\tilde{\varepsilon}$ satisfies the orthogonal conditions
\[
\left(\tilde{\varepsilon},i\Lambda P\right)_2=\left(\tilde{\varepsilon},|\cdot|^2P\right)_2=\left(\tilde{\varepsilon},i\rho\right)_2=0
\]
in $I$. In particular, $\tilde{\lambda}$ and $\tilde{b}$ are unique within functions and $\tilde{\gamma}$ is unique within continuous functions (and is unique within functions under modulo $2\pi$).
\end{lemma}

See \cite{Npote,Ninv} for the proof.

\section{Approximate blow-up law}
In this section, we describe the initial values and the approximation functions of the parameters $\lambda$ and $b$ in the decomposition.

\begin{lemma}
Let
\[
\lambda_{\mathrm{app}}(s)=\left(\frac{\alpha}{2}\sqrt{\frac{2\beta}{2-\alpha}}\right)^{-\frac{2}{\alpha}}s^{-\frac{2}{\alpha}},\quad b_{\mathrm{app}}(s)=\frac{2}{\alpha s}.
\]
Then, $(\lambda_{\mathrm{app}},b_{\mathrm{app}})$ is solutions of
\[
\frac{\partial b}{\partial s}+b^2-\beta\lambda^\alpha=0,\quad b+\frac{1}{\lambda}\frac{\partial \lambda}{\partial s}=0
\]
in $s>0$.
\end{lemma}

\begin{lemma}[\cite{LMR,Ninv}]
\label{paraini}
Let define $C_0:=\frac{8E_0}{\||y|Q\|_2^2}$ and $0<\lambda_0\ll 1$ such that $\frac{2\beta}{2-\alpha}+C_0{\lambda_0}^{2-\alpha}>0$. For $\lambda\in(0,\lambda_0]$, we set
\[
\mathcal{F}(\lambda):=\int_\lambda^{\lambda_0}\frac{1}{\mu^{\frac{\alpha}{2}+1}\sqrt{\frac{2\beta}{2-\alpha}+C_0\mu^{2-\alpha}}}d\mu.
\]
Then, for any $s_1\gg 1$, there exist $b_1,\lambda_1>0$ such that
\[
\left|\frac{{\lambda_1}^{\frac{\alpha}{2}}}{\lambda_{\mathrm{app}}(s_1)^{\frac{\alpha}{2}}}-1\right|+\left|\frac{b_1}{b_{\mathrm{app}}(s_1)}-1\right|\lesssim {s_1}^{-\frac{1}{2}}+{s_1}^{2-\frac{4}{\alpha}},\quad \mathcal{F}(\lambda_1)=s_1,\quad E(P_{\lambda_1,b_1,\gamma})=E_0.
\]
Moreover,
\[
\left|\mathcal{F}(\lambda)-\frac{2}{\alpha\lambda^{\frac{\alpha}{2}}\sqrt{\frac{2\beta}{2-\alpha}}}\right|\lesssim\lambda^{-\frac{\alpha}{4}}+\lambda^{2-\frac{3}{2}\alpha}
\]
holds.
\end{lemma}

\begin{proof}
See \cite{LMR,Ninv} for the proof. The proof is dimension-independent.
\end{proof}

\section{Uniformity estimates for decomposition}
In this section, we estimate \textit{modulation terms}.

Let define
\[
\mathcal{C}:=\frac{\alpha}{4-\alpha}\left(\frac{\alpha}{2}\sqrt{\frac{2\beta}{2-\alpha}}\right)^{-\frac{4}{\alpha}}.
\]
For $t_1<0$ that is sufficiently close to $0$, we define
\[
s_1:=|\mathcal{C}^{-1}t_1|^{-\frac{\alpha}{4-\alpha}}.
\]
Additionally, let $\lambda_1$ and $b_1$ be given in Lemma \ref{paraini} for $s_1$ and $\gamma_1=0$. Let $u$ be the solution for (NLS$+$) with an initial value
\begin{align}
\label{initial}
u(t_1,x):=P_{\lambda_1,b_1,0}(x).
\end{align}
Then, since $u$ satisfies the assumption of Lemma \ref{decomposition} in a neighbourhood of $t_1$, there exists a decomposition $(\tilde{\lambda}_{t_1},\tilde{b}_{t_1},\tilde{\gamma}_{t_1},\tilde{\varepsilon}_{t_1})$ such that $(\ref{mod})$ in a neighbourhood $I$ of $t_1$. The rescaled time $s_{t_1}$ is defined as
\[
s_{t_1}(t):=s_1-\int_t^{t_1}\frac{1}{\tilde{\lambda}_{t_1}(\tau)^2}d\tau.
\]
Therefore, we define an inverse function ${s_{t_1}}^{-1}:s_{t_1}(I)\rightarrow I$. Therefore, we define
\begin{align*}
&t_{t_1}:={s_{t_1}}^{-1},\quad \lambda_{t_1}(s):=\tilde{\lambda}(t_{t_1}(s)),\quad b_{t_1}(s):=\tilde{b}(t_{t_1}(s)),\\
&\gamma_{t_1}(s):=\tilde{\gamma}(t_{t_1}(s)),\quad \varepsilon_{t_1}(s,y):=\tilde{\varepsilon}(t_{t_1}(s),y).
\end{align*}
If there is no risk of confusion, the subscript $t_1$ is omitted. In particular, it should be noted that $u\in C((-T_*,T^*),\Sigma^2(\mathbb{R}^N))$ and $|x|\nabla u\in C((-T_*,T^*),L^2(\mathbb{R}^N))$. Furthermore, let $I_{t_1}$ be the maximal interval such that a decomposition as $(\ref{mod})$ is obtained and $J_{s_1}:=s\left(I_{t_1}\right)$. Additionally, let $s_0\ (\leq s_1)$ be sufficiently large and let $s':=\max\left\{s_0,\inf J_{s_1}\right\}$.

Let $0<M<\min\{\frac{1}{2},\frac{4}{\alpha}-2\}$ and $s_*$ be defined as
\[
s_*:=\inf\left\{\sigma\in(s',s_1]\ |\ \mbox{(\ref{bootstrap}) holds on }[\sigma,s_1]\right\},
\]
where
\begin{align}
\label{bootstrap}
\left\|\varepsilon(s)\right\|_{H^1}^2+b(s)^2\||\cdot|\varepsilon(s)\|_2^2<s^{-2K},\quad \left|\frac{\lambda(s)^{\frac{\alpha}{2}}}{\lambda_{\mathrm{app}}(s)^{\frac{\alpha}{2}}}-1\right|+\left|\frac{b(s)}{b_{\mathrm{app}}(s)}-1\right|<s^{-M}.\\
\end{align}

Finally, we define
\[
\Mod:=\left(\frac{1}{\lambda}\frac{\partial \lambda}{\partial s}+b,\frac{\partial b}{\partial s}+b^2-\theta,1-\frac{\partial \gamma}{\partial s}\right).
\]

In the following discussion, the constant $\epsilon>0$ is a sufficiently small constant. If necessary, $s_0$ and $s_1$ are recalculated in response to $\epsilon>0$.

\begin{lemma}[The equation for $\varepsilon$]
In $J_{s_1}$, 
\begin{align}
\label{epsieq}
&i\frac{\partial \varepsilon}{\partial s}+\Delta \varepsilon-\varepsilon+f\left(P+\varepsilon\right)-f\left(P\right)-\lambda^\alpha \left(g(P+\varepsilon)-g(P)\right)\nonumber\\
&\hspace{20pt}-i\left(\frac{1}{\lambda}\frac{\partial \lambda}{\partial s}+b\right)\Lambda (P+\varepsilon)+\left(1-\frac{\partial \gamma}{\partial s}\right)(P+\varepsilon)+\left(\frac{\partial b}{\partial s}+b^2-\theta\right)\frac{|y|^2}{4}(P+\varepsilon)-\left(\frac{1}{\lambda}\frac{\partial \lambda}{\partial s}+b\right)b\frac{|y|^2}{2}(P+\varepsilon)\nonumber\\
=&\Psi
\end{align}
holds.
\end{lemma}

\begin{lemma}
For $s\in(s_*,s_1]$,
\[
\left|(\varepsilon(s),Q)\right|\lesssim s^{-(K+2)},\quad \left|\Mod(s)\right|\lesssim s^{-(K+2)},\quad \left\|e^{\epsilon'|\cdot|}\Psi\right\|_{H^1}\lesssim s^{-(K+4)}
\]
hold.
\end{lemma}

\begin{proof}
We outline the proof. See \cite{LMR,Ninv} for detail of the proof.

Let
\[
s_{**}:=\inf\left\{\ s\in[s_*,s_1]\ \middle|\ |(\varepsilon(\tau),P)_2|<\tau^{-(K+2)}\ \mbox{holds on}\ [s,s_1].\ \right\}.
\]
We work on the interval $[s_{**},s_1]$.

According to the orthogonality properties, we have
\begin{align}
\label{dortho1}
0&=\frac{d}{ds}\left(i\varepsilon,\Lambda P\right)_2=\left(i\frac{\partial \varepsilon}{\partial s},\Lambda P\right)_2+\left(i\varepsilon,\frac{\partial (\Lambda P)}{\partial s}\right)_2\\
\label{dortho2}
&=\frac{d}{ds}\left(i\varepsilon,i|\cdot|^2 P\right)_2=\left(i\frac{\partial \varepsilon}{\partial s},i|\cdot|^2 P\right)_2+\left(i\varepsilon,i|\cdot|^2 \frac{\partial P}{\partial s}\right)_2\\
\label{dortho3}
&=\frac{d}{ds}\left(i\varepsilon,\rho\right)_2=\left(i\frac{\partial \varepsilon}{\partial s},\rho\right)_2.
\end{align}

For (\ref{dortho1}), we have
\[
\left(i\varepsilon,\frac{\partial (\Lambda P)}{\partial s}\right)_2=O(s^{-(K+3)})+O(s^{-1}|\Mod(s)|).
\]
Moreover, we have
\begin{align*}
&\left(i\frac{\partial \varepsilon}{\partial s},\Lambda P\right)_2\\
=&\left(L_+\Re\varepsilon+iL_-\Im\varepsilon-\left(f\left(P+\varepsilon\right)-f\left(P\right)-df(Q)(\varepsilon)\right)+\lambda^\alpha \left(g(P+\varepsilon)-g(P)\right)\right.\\
&\hspace{10pt}\left.+i\left(\frac{1}{\lambda}\frac{\partial \lambda}{\partial s}+b\right)\Lambda (P+\varepsilon)-\left(1-\frac{\partial \gamma}{\partial s}\right)(P+\varepsilon)-\left(\frac{\partial b}{\partial s}+b^2-\theta\right)\frac{|\cdot|^2}{4}(P+\varepsilon)+\left(\frac{1}{\lambda}\frac{\partial \lambda}{\partial s}+b\right)b\frac{|\cdot|^2}{2}(P+\varepsilon)+\Psi,\Lambda P\right)_2.
\end{align*}
Here, we have
\begin{align*}
&\left|\left(L_+\Re\varepsilon,\Lambda Q\right)_2\right|+\left|\left(L_+\Re\varepsilon,\lambda^\alpha\Lambda Z\right)_2\right|+\left|\left(iL_-\Im\varepsilon,\Lambda P\right)_2\right|+\left|\lambda^\alpha\left(g(P+\varepsilon)-g(P),\Lambda P\right)_2\right|+\left|\left(\Psi,\Lambda P\right)_2\right|\\
=&O(s^{-(K+2)}+s^{-1}|\Mod(s)|)
\end{align*}
and
\[
\left(|\cdot|^2P,\Lambda P\right)_2=-\||\cdot|Q\|_2^2+O(s^{-2}).
\]
Moreover, we have
\[
f\left(P+\varepsilon\right)-f\left(P\right)-df(Q)(\varepsilon)=f\left(P+\varepsilon\right)-f\left(P\right)-df(P)(\varepsilon)+df(P)(\varepsilon)-df(Q)(\varepsilon).
\]
We prove only the case $N\geq 4$. If $Q<3|\lambda^\alpha Z|$, then we obtain
\[
\left|\left(f(P+\varepsilon)-f(P)-df(P)(\varepsilon)\right)\Lambda\overline{P}\right|\lesssim \lambda^{\alpha}(1+|\cdot|^\kappa)(Q^{\frac{4}{N}}+|\varepsilon|^{\frac{4}{N}})|\varepsilon|Q
\]
since $1\lesssim \lambda^\alpha(1+|\cdot|)$. If $3|\lambda^\alpha Z|\leq Q$ and $Q<3|\varepsilon|$, then we obtain
\[
\left|\left(f(P+\varepsilon)-f(P)-df(P)(\varepsilon)\right)\Lambda\overline{P}\right|\lesssim (1+|\cdot|^\kappa)Q^{\frac{4}{N}}|\varepsilon|^2.
\]
If $3|\varepsilon|\leq Q$, then $P-|\varepsilon|>\frac{1}{3}Q>0$. We have
\[
\left|\left(f(P+\varepsilon)-f(P)-df(P)(\varepsilon)\right)\Lambda\overline{P}\right|\lesssim(1+|\cdot|^\kappa)Q^{\frac{4}{N}}|\varepsilon|^2.
\]
Therefore, we have
\[
\left(f(P+\varepsilon)-f(P)-df(P)(\varepsilon),\Lambda P\right)_2=O(s^{-(K+2)}).
\]
Similarly, for $\left(df(P)(\varepsilon)-df(Q)(\varepsilon)\right)\Lambda\overline{P}$, we have
\[
\left(df(P)(\varepsilon)-df(Q)(\varepsilon),\Lambda P\right)_2=O(s^{-(K+2)}).
\]
Accordingly, we have
\[
\left(i\frac{\partial \varepsilon}{\partial s},\Lambda P\right)_2=-\frac{1}{4}\||\cdot|Q\|\left(\frac{\partial b}{\partial s}+b^2-\theta\right)+O(s^{-(K+2)})+O(s^{-1}|\Mod(s)|)
\]
and
\[
\frac{\partial b}{\partial s}+b^2-\theta=O(s^{-(K+2)})+O(s^{-1}|\Mod(s)|).
\]

The same calculations for (\ref{dortho2}) and (\ref{dortho3}) yield
\[
\frac{1}{\lambda}\frac{\partial \lambda}{\partial s}+b=O(s^{-(K+2)})+O(s^{-1}|\Mod(s)|),\quad 1-\frac{\partial \gamma}{\partial s}=O(s^{-(K+2)})+O(s^{-1}|\Mod(s)|).
\]
Consequently, we have
\[
\left|\Mod(s)\right|\lesssim s^{-(K+2)},\quad \left\|e^{\epsilon|y|}\Psi\right\|_{H^1}\lesssim s^{-(K+4)}.
\]

The rest of the proof is the same as the proof in \cite{LMR,Ninv}.
\end{proof}

\section{Modified energy function}
\label{sec:MEF}
In this section, we proceed with a modified version \cite{Npote,Ninv} of the technique presented in Le Coz, Martel, Rapha\"{e}l \cite{LMR} and Martel and Szeftel \cite{RSEU}. Let $m>0$ be sufficiently large and define
\begin{align*}
H(s,\varepsilon)&:=\frac{1}{2}\left\|\varepsilon\right\|_{H^1}^2+b^2\left\||y|\varepsilon\right\|_2^2-\int_{\mathbb{R}^N}\left(F(P+\varepsilon)-F(P)-dF(P)(\varepsilon)\right)dy\\
&\hspace{60pt}-\lambda^\alpha\int_{\mathbb{R}^N}\left(G(P+\varepsilon)-G(P)-dG(P)(\varepsilon)\right)dy,\\
S(s,\varepsilon)&:=\frac{1}{\lambda^m}H(s,\varepsilon).
\end{align*}

\begin{lemma}[Coercivity of $H$]
\label{Hcoer}
For $s\in(s_*,s_1]$, 
\[
\|\varepsilon\|_{H^1}^2+b^2\left\||y|\varepsilon\right\|_2^2+O(s^{-2(K+2)})\lesssim H(s,\varepsilon)\lesssim \|\varepsilon\|_{H^1}^2+b^2\left\||y|\varepsilon\right\|_2^2
\]
hold.
\end{lemma}

\begin{proof}
We prove only the case $N\geq 4$.

If $2|\varepsilon|\geq |P|$, then we have
\[
\left|F(P+\varepsilon)-F(P)-dF(P)(\varepsilon)-\frac{1}{2}d^2F(P)(\varepsilon,\varepsilon)\right|\lesssim |\varepsilon|^{\frac{4}{N}+2}.
\]
If $2|\varepsilon|<|P|$, then we have
\[
\left|F(P+\varepsilon)-F(P)-dF(P)(\varepsilon)-\frac{1}{2}d^2F(P)(\varepsilon,\varepsilon)\right|\lesssim \left(|P|-|\varepsilon|\right)^{\frac{4}{N}-1}|\varepsilon|^3\lesssim |\varepsilon|^{\frac{4}{N}+2}.
\]
Therefore, we obtain
\[
\int_{\mathbb{R}^N}\left(F(P(y)+\varepsilon(y))-F(P(y))-dF(P(y))(\varepsilon(y))-\frac{1}{2}d^2F(P(y))(\varepsilon(y),\varepsilon(y))\right)dy=o(\|\varepsilon\|_{H^1}^2).
\]

If $2|\lambda^\alpha Z|\geq Q$, then we have
\[
\left|\frac{1}{2}d^2F(P)(\varepsilon,\varepsilon)-\frac{1}{2}d^2F(Q)(\varepsilon,\varepsilon)\right|\lesssim |\lambda^\alpha Z|^{\frac{4}{N}}|\varepsilon|^2.
\]
If $2|\lambda^\alpha Z|<Q$, then we have
\[
\left|\frac{1}{2}d^2F(P)(\varepsilon,\varepsilon)-\frac{1}{2}d^2F(Q)(\varepsilon,\varepsilon)\right|\lesssim \lambda^\alpha\left(Q-|\lambda^\alpha Z|\right)^{\frac{4}{N}-1}|\varepsilon|^2|Z|\lesssim (1+|\cdot|^\kappa)\lambda^\alpha|\varepsilon|^2Q^\frac{4}{N}.
\]
Therefore, we obtain
\[
\int_{\mathbb{R}^N}\left(\frac{1}{2}d^2F(P(y))(\varepsilon(y),\varepsilon(y))-\frac{1}{2}d^2F(Q)(\varepsilon(y),\varepsilon(y))\right)dy=o(\|\varepsilon\|_{H^1}^2).
\]

Moreover, we have
\[
\int_{\mathbb{R}^N}\left(G(P(y)+\varepsilon(y))-G(P(y))-dG(P(y))(\varepsilon(y))\right)dy=O\left(\|\varepsilon\|_{H^1}^2\right).
\]

Finally. we have
\begin{align*}
\left\|\varepsilon\right\|_{H^1}^2-\int_{\mathbb{R}^N}d^2F(Q)(\varepsilon(y),\varepsilon(y))dy&=\left\langle L_+\Re\varepsilon,\Re\varepsilon\right\rangle+\left\langle L_-\Im\varepsilon,\Im\varepsilon\right\rangle\\
&\geq\mu\|\varepsilon\|_{H^1}^2-\frac{1}{\mu}\left((\Re\varepsilon,Q)_2^2+(\Re\varepsilon,|\cdot|^2 Q)_2^2+(\Im\varepsilon,\rho)_2^2\right)\\
&=\mu\|\varepsilon\|_{H^1}^2+O(s^{-2(K+2)}).
\end{align*}
\end{proof}

\begin{corollary}[Estimation of $S$]
\label{Sesti}
For $s\in(s_*,s_1]$, 
\[
\frac{1}{\lambda^m}\left(\|\varepsilon\|_{H^1}^2+b^2\left\||y|\varepsilon\right\|_2^2+O(s^{-2(K+2)})\right)\lesssim S(s,\varepsilon)\lesssim \frac{1}{\lambda^m}\left(\|\varepsilon\|_{H^1}^2+b^2\left\||y|\varepsilon\right\|_2^2\right)
\]
hold.
\end{corollary}

\begin{lemma}
\label{Lambda}
For $s\in(s_*,s_1]$ and $0\leq q\leq  \frac{4}{N-4}\ (N\geq 5)$ and $0\leq q<\infty\ (N\leq4)$, 
\begin{align}
\label{flambdaesti}
\left|\left(\left|P+\varepsilon\right|^q(P+\varepsilon)-|P|^qP,\Lambda \varepsilon\right)_2\right|&\lesssim \|\varepsilon\|_{H^1}^2+s^{-3K}
\end{align}
holds.
\end{lemma}

\begin{proof}
If $q=0$, then the lemma holds clearly. Therefore, we may $p\neq0$.

Let
\[
j(z)=|z|^qz,\quad J(z)=\frac{1}{q+2}|z|^{q+2}.
\]

Calculated in the same way as in Section 5.4 in \cite{LMR}, we have
\begin{align*}
&\nabla\left(J(P+\varepsilon)-J(P)-dJ(P)(\varepsilon)\right)\\
=&\Re\left(j(P+\varepsilon)\nabla\left(\overline{P}+\overline{\varepsilon}\right)-j(P)\nabla\overline{P}-dj(P)(\varepsilon)\nabla\overline{P}-j(P)\nabla\overline{\varepsilon}\right)\\
=&\Re\left(\left(j(P+\varepsilon)-j(P)-dj(P)(\varepsilon)\right)\nabla\overline{P}+\left(j(P+\varepsilon)-j(P)\right)\nabla\overline{\varepsilon}\right).
\end{align*}
Therefore, we have
\begin{align*}
&\left(j(P+\varepsilon)-j(P),\Lambda \varepsilon\right)=\Re\int_{\mathbb{R}^N}\left(j(P+\varepsilon)-j(P)\right)\Lambda\overline{\varepsilon}dy\\
=&\Re\int_{\mathbb{R}^N}\left(j(P+\varepsilon)-j(P)\right)\left(\frac{N}{2}\overline{\varepsilon}+y\cdot\nabla\overline{\varepsilon}\right)dy\\
=&\Re\int_{\mathbb{R}^N}\bigg(\frac{N}{2}\left(j(P+\varepsilon)-j(P)\right)\overline{\varepsilon}-y\cdot\left(\left(j(P+\varepsilon)-j(P)-dj(P)(\varepsilon)\right)\nabla\overline{P}\right.\\
&\hspace{60pt}\left.+\nabla\left(J(P+\varepsilon)-J(P)-dJ(P)(\varepsilon)\right)\right)\bigg)dy\\
=&\Re\int_{\mathbb{R}^N}\bigg(\frac{N}{2}\left(j(P+\varepsilon)-j(P)\right)\overline{\varepsilon}-\left(j(P+\varepsilon)-j(P)-dj(P)(\varepsilon)\right)y\cdot\nabla\overline{P}\\
&\hspace{60pt}-N\left(J(P+\varepsilon)-J(P)-dJ(P)(\varepsilon)\right)\bigg)dy.
\end{align*}
Firstly,
\[
\left|\left(j(P+\varepsilon)-j(P)\right)\overline{\varepsilon}\right|+\left|J(P+\varepsilon)-J(P)-dJ(P)(\varepsilon)\right|\lesssim((1+|\cdot|^\kappa)Q^q+|\varepsilon|^q)|\varepsilon|^2
\]
holds.

Next, we consider $\left(j(P+\varepsilon)-j(P)-dj(P)(\varepsilon)\right)y\cdot\nabla\overline{P}$. If $q>1$, then we have
\[
\left|\left(j(P+\varepsilon)-j(P)-dj(P)(\varepsilon)\right)y\cdot\nabla\overline{P}\right|\lesssim(1+|\cdot|^\kappa)(Q+|\varepsilon|)^{q-1}|\varepsilon|^2Q.
\]
On the other hands, we assume $q\leq 1$. If $Q<3|\lambda^\alpha Z|$, then we have
\[
\left|\left(j(P+\varepsilon)-j(P)-dj(P)(\varepsilon)\right)y\cdot\nabla\overline{P}\right|\lesssim \lambda^{K\alpha}(1+|\cdot|^\kappa)(Q^q+|\varepsilon|^q)|\varepsilon|Q
\]
since $1\lesssim \lambda^\alpha(1+|\cdot|)$.

If $3|\lambda^\alpha Z|\leq Q$ and $Q<3|\varepsilon|$, then we have
\[
\left|\left(j(P+\varepsilon)-j(P)-dj(P)(\varepsilon)\right)y\cdot\nabla\overline{P}\right|\lesssim (1+|\cdot|^\kappa)Q^q|\varepsilon|^2.
\]
If $3|\varepsilon|\leq Q$, then we have
\[
\left|\left(j(P+\varepsilon)-j(P)-dj(P)(\varepsilon)\right)y\cdot\nabla\overline{P}\right|\lesssim(1+|\cdot|^\kappa)Q^q|\varepsilon|^2.
\]
\end{proof}

\begin{lemma}[Derivative of $H$ in time]
\label{Hdiff}
For $s\in(s_*,s_1]$, 
\[
\frac{d}{ds}H(s,\varepsilon(s))\geq -Cb\left(\|\varepsilon\|_{H^1}^2+b^2\left\||y|\varepsilon\right\|_2^2\right)+O(s^{-2(K+2)})
\]
holds.
\end{lemma}

\begin{proof}
We prove the lemma by combining Lemma \ref{Lambda} and the proofs in \cite{LMR,Npote,Ninv}.
\end{proof}

\begin{lemma}[Derivative of $S$ in time]
\label{Sdiff}
Let $m>0$ be sufficiently large. Then,
\[
\frac{d}{ds}S(s,\varepsilon(s))\gtrsim \frac{b}{\lambda^m}\left(\|\varepsilon\|_{H^1}^2+b^2\left\||\cdot|\varepsilon\right\|_2^2+O(s^{-(2K+3)})\right)
\]
holds for $s\in(s_*,s_1]$.
\end{lemma}

\begin{proof}
See \cite{LMR,Ninv} for the proof.
\end{proof}

\section{Bootstrap}
In this section, we use the estimates obtained in Section \ref{sec:MEF} and the bootstrap to establish the estimates of the parameters. However, we introduce the following lemmas without proofs. Regarding lemmas in this section, see \cite{LMR,Ninv} for the proof.

\label{sec:bootstrap}
\begin{lemma}[Re-estimation]
\label{rebootstrap}
For $s\in(s_*,s_1]$, 
\begin{align}
\label{reepsiesti}
\left\|\varepsilon(s)\right\|_{H^1}^2+b(s)^2\left\||\cdot|\varepsilon(s)\right\|_2^2&\lesssim s^{-(2K+2)},\\
\label{reesti}
\left|\frac{\lambda(s)^{\frac{\alpha}{2}}}{\lambda_{\mathrm{app}}(s)^{\frac{\alpha}{2}}}-1\right|+\left|\frac{b(s)}{b_{\mathrm{app}}(s)}-1\right|&\lesssim s^{-\frac{1}{2}}+s^{2-\frac{4}{\alpha}}
\end{align}
holds.
\end{lemma}

\begin{corollary}
\label{reesti}
If $s_0$ is sufficiently large, then $s_*=s'$.
\end{corollary}

\begin{lemma}
If $s_0$ is sufficiently large, then $s'=s_0$.
\end{lemma}

\begin{lemma}[Interval]
\label{interval}
If $s_0$ is sufficiently large, then there is a $t_0<0$ that is sufficiently close to $0$ such that for $t_1\in(t_0,0)$, 
\[
[t_0,t_1]\subset {s_{t_1}}^{-1}([s_0,s_1]),\quad \left|\mathcal{C}s_{t_1}(t)^{-\frac{4-\alpha}{\alpha}}-|t|\right|\lesssim |t|^{1+\frac{\alpha M}{4-\alpha}}\ (t\in [t_0,t_1])
\]
holds.
\end{lemma}

\begin{lemma}[Conversion of estimates]
\label{uniesti}
Let
\[
\mathcal{C}_\lambda:=\mathcal{C}^{-\frac{2}{4-\alpha}}\left(\frac{\alpha}{2}\sqrt{\frac{2\beta}{2-\alpha}}\right)^{-\frac{2}{\alpha}},\quad \mathcal{C}_b:=\frac{2}{\alpha}\mathcal{C}^{-\frac{\alpha}{4-\alpha}}.
\]
For $t\in[t_0,t_1]$, 
\begin{align*}
\tilde{\lambda}_{t_1}(t)=\mathcal{C}_\lambda|t|^\frac{2}{4-\alpha}\left(1+\epsilon_{\tilde{\lambda},t_1}(t)\right),\quad \tilde{b}_{t_1}(t)=\mathcal{C}_b|t|^\frac{\alpha}{4-\alpha}\left(1+\epsilon_{\tilde{b},t_1}(t)\right),\quad \|\tilde{\varepsilon}_{t_1}(t)\|_{H^1}\lesssim |t|^{\frac{\alpha K}{4-\alpha}},\quad \||\cdot|\tilde{\varepsilon}_{t_1}(t)\|_2\lesssim |t|^{\frac{\alpha (K-1)}{4-\alpha}}
\end{align*}
holds. Furthermore,
\[
\sup_{t_1\in[t,0)}\left|\epsilon_{\tilde{\lambda},t_1}(t)\right|\lesssim |t|^\frac{\alpha M}{4-\alpha},\ \sup_{t_1\in[t,0)}\left|\epsilon_{\tilde{b},t_1}(t)\right|\lesssim |t|^\frac{\alpha M}{4-\alpha}.
\]
\end{lemma}

\section{Proof of the main result}
See \cite{LMR,Npote} for details of the proof.

\begin{proof}
Let $(t_n)_{n\in\mathbb{N}}\subset(t_0,0)$ be a monotonically increasing sequence such that $\lim_{n\nearrow \infty}t_n=0$. For each $n\in\mathbb{N}$, $u_n$ is the solution for (NLS$+$) with an initial value
\begin{align*}
u_n(t_n,x):=P_{\lambda_{1,n},b_{1,n},0}(x)
\end{align*}
at $t_n$, where $b_{1,n}$ and $\lambda_{1,n}$ are given by Lemma \ref{paraini} for $t_n$.

According to Lemma \ref{decomposition} with an initial value $\tilde{\gamma}_n(t_n)=0$ on $[t_0,t_1]$, there exists a decomposition
\[
u_n(t,x)=\frac{1}{\tilde{\lambda}_n(t)^{\frac{N}{2}}}\left(P+\tilde{\varepsilon}_n\right)\left(t,\frac{x}{\tilde{\lambda}_n(t)}\right)e^{-i\frac{\tilde{b}_n(t)}{4}\frac{|x|^2}{\tilde{\lambda}_n(t)^2}+i\tilde{\gamma}_n(t)}.
\]
Then, $(u_n(t_0))_{n\in\mathbb{N}}$ is bounded in $\Sigma^1$. Therefore, up to a subsequence, there exists $u_\infty(t_0)\in \Sigma^1$ such that
\[
u_n(t_0)\rightharpoonup u_\infty(t_0)\quad \mathrm{in}\ \Sigma^1,\quad u_n(t_0)\rightarrow u_\infty(t_0)\quad \mathrm{in}\ L^2(\mathbb{R}^N)\quad (n\rightarrow\infty),
\]
see \cite{LMR,Npote} for details.

Let $u_\infty$ be the solution for (NLS$+$) with an initial value $u_\infty(t_0)$ and $T^*$ be the supremum of the maximal existence interval of $u_\infty$. Moreover, we define $T:=\min\{0,T^*\}$. Then, for any $T'\in[t_0,T)$, $[t_0,T']\subset[t_0,t_n]$ if $n$ is sufficiently large. Then, there exist $n_0$ and $C(T',t_0)>0$ such that 
\[
\sup_{n\geq n_0}\|u_n\|_{L^\infty([t_0,T'],\Sigma^1)}\leq C(T',t_0)
\]
holds. According to Lemma B.2 in \cite{Npote}, 
\[
u_n\rightarrow u_\infty\quad \mathrm{in}\ C\left([t_0,T'],L^2(\mathbb{R}^N)\right)\quad (n\rightarrow\infty)
\]
holds. In particular, $u_n(t)\rightharpoonup u_\infty(t)\ \mathrm{in}\ \Sigma^1$ for any $t\in [t_0,T)$. Furthermore, from the mass conservation, we have
\[
\|u_\infty(t)\|_2=\|u_\infty(t_0)\|_2=\lim_{n\rightarrow\infty}\|u_n(t_0)\|_2=\lim_{n\rightarrow\infty}\|u_n(t_n)\|_2=\lim_{n\rightarrow\infty}\|P(t_n)\|_2=\|Q\|_2.
\]

Based on weak convergence in $H^1(\mathbb{R}^N)$ and Lemma \ref{decomposition}, we decompose $u_\infty$ to
\[
u_\infty(t,x)=\frac{1}{\tilde{\lambda}_\infty(t)^{\frac{N}{2}}}\left(P+\tilde{\varepsilon}_\infty\right)\left(t,\frac{x}{\tilde{\lambda}_\infty(t)}\right)e^{-i\frac{\tilde{b}_\infty(t)}{4}\frac{|x|^2}{{\tilde{\lambda}_\infty(t)}^2}+i\tilde{\gamma}_\infty(t)},
\]
where an initial value of $\tilde{\gamma}_\infty$ is $\gamma_\infty(t_0)\in\left(|t_0|^{-1}-\pi,|t_0|^{-1}+\pi\right]\cap\tilde{\gamma}(u_\infty(t_0))$ (which is unique, see \cite{Npote}). Furthermore, for any $t\in[t_0,T)$, as $n\rightarrow\infty$, 
\[
\tilde{\lambda}_n(t)\rightarrow\tilde{\lambda}_\infty(t),\quad \tilde{b}_n(t)\rightarrow \tilde{b}_\infty(t),\quad e^{i\tilde{\gamma}_n(t)}\rightarrow e^{i\tilde{\gamma}_\infty(t)},\quad\tilde{\varepsilon}_n(t)\rightharpoonup \tilde{\varepsilon}_\infty(t)\quad \mathrm{in}\ \Sigma^1
\]
holds. Consequently, for a uniform estimate of Lemma \ref{uniesti}, as $n\rightarrow\infty$, we have
\begin{align*}
&\tilde{\lambda}_{\infty}(t)=\mathcal{C}_\lambda\left|t\right|^\frac{2}{4-\alpha}(1+\epsilon_{\tilde{\lambda},0}(t)),\quad \tilde{b}_{\infty}(t)=\mathcal{C}_b\left|t\right|^\frac{\alpha}{4-\alpha}(1+\epsilon_{\tilde{b},0}(t)),\\
&\|\tilde{\varepsilon}_{\infty}(t)\|_{H^1}\lesssim \left|t\right|^{\frac{\alpha K}{4-\alpha}},\quad \||y|\tilde{\varepsilon}_{\infty}(t)\|_2\lesssim \left|t\right|^{\frac{\alpha (K-1)}{4-\alpha}},\quad\left|\epsilon_{\tilde{\lambda},0}(t)\right|\lesssim |t|^\frac{\alpha M}{4-\alpha},\quad \left|\epsilon_{\tilde{b},0}(t)\right|\lesssim |t|^{\frac{\alpha M}{4-\alpha}}.
\end{align*}
Consequently, we obtain that $u$ converge to the blow-up profile in $\Sigma^1$.

Finally, we check energy of $u_\infty$. Since
\[
E\left(u_n\right)-E\left(P_{\tilde{\lambda}_n,\tilde{b}_n,\tilde{\gamma}_n}\right)=\int_0^1\left\langle E'(P_{\tilde{\lambda}_n,\tilde{b}_n,\tilde{\gamma}_n}+\tau \tilde{\varepsilon}_{\tilde{\lambda}_n,\tilde{b}_n,\tilde{\gamma}_n}),\tilde{\varepsilon}_{\tilde{\lambda}_n,\tilde{b}_n,\tilde{\gamma}_n}\right\rangle d\tau
\]
and $E'(w)=-\Delta w-|w|^\frac{4}{N}w-|y|^{-2\sigma}w$, we have
\[
E\left(u_n\right)-E\left(P_{\tilde{\lambda}_n,\tilde{b}_n,\tilde{\gamma}_n}\right)=O\left(\frac{1}{{\tilde{\lambda}_n}^2}\|\tilde{\varepsilon}_n\|_{H^1}\right)=O\left(|t|^\frac{\alpha K-4}{4-\alpha}\right).
\]
Similarly, we have
\[
E\left(u_\infty\right)-E\left(P_{\tilde{\lambda}_\infty,\tilde{b}_\infty,\tilde{\gamma}_\infty}\right)=O\left(\frac{1}{{\tilde{\lambda}_\infty}^2}\|\tilde{\varepsilon}_\infty\|_{H^1}\right)=O\left(|t|^\frac{\alpha K-4}{4-\alpha}\right).
\]
From continuity of $E$, we have
\[
\lim_{n\rightarrow \infty}E\left(P_{\tilde{\lambda}_n,\tilde{b}_n,\tilde{\gamma}_n}\right)=E\left(P_{\tilde{\lambda}_\infty,\tilde{b}_\infty,\tilde{\gamma}_\infty}\right)
\]
and from the conservation of energy,
\[
E\left(u_n\right)=E\left(u_n(t_n)\right)=E\left(P_{\tilde{\lambda}_{1,n},\tilde{b}_{1,n},\tilde{\gamma}_{1,n}}\right)=E_0.
\]
Therefore, we have
\[
E\left(u_\infty\right)=E_0+o_{t\nearrow0}(1)
\]
and since $E\left(u_\infty\right)$ is constant for $t$, $E\left(u_\infty\right)=E_0$.
\end{proof}


\begin{thebibliography}{99}
\bibitem{BCD} V. Banica, R. Carles and T. Duyckaerts. Minimal blow-up solutions to the mass-critical inhomogeneous NLS equation. \textit{Comm. Partial Differential Equations} 36 (2011), no. 3, 487–531.
\bibitem{BLGS} H. Berestycki and P.-L. Lions. Nonlinear scalar field equations. I. Existence of a ground state. \textit{Arch. Rational Mech. Anal.} 82 (1983), no. 4, 313–345.
\bibitem{CSSE} T. Cazenave. \textit{Semilinear Schr\"{o}dinger equations.} Courant Lecture Notes in Mathematics, 10. New York University, Courant Institute of Mathematical Sciences, New York; American Mathematical Society, Providence, RI, 2003.
\bibitem{EF} E. Csobo and F. Genoud. Minimal mass blow-up solutions for the $L^2$ critical NLS with inverse-square potential. \textit{Nonlinear Anal}. 168 (2018), 110–129.
\bibitem{GT} D. Gilbarg and N. S. Trudinger. \textit{Elliptic partial differential equations of second order. Second edition}. Grundlehren der Mathematischen Wissenschaften [Fundamental Principles of Mathematical Sciences], 224. Springer-Verlag, Berlin, 1983.
\bibitem{KGS} M. K. Kwong. Uniqueness of positive solutions of $\delta u-u+u^p=0\ \mathrm{in}\ \mathbb{R}^n$. \textit{Arch. Rational Mech. Anal.} 105 (1989), no. 3, 243–266.
\bibitem{LMR} S. Le Coz, Y. Martel and P. Rapha\"{e}l. Minimal mass blow up solutions for a double power nonlinear Schr\"{o}dinger equation. \textit{Rev. Mat. Iberoam.} 32 (2016), no. 3, 795–833.
\bibitem{Npote} N. Matsui. Minimal mass blow-up solutions for nonlinear Schr\"{o}dinger equations with a potential, arXiv preprint \url{https://arxiv.org/abs/2007.15968}
\bibitem{Ninv} N. Matsui. Minimal mass blow-up solutions for nonlinear Schr\"{o}dinger equations with an inverse potential, arXiv preprint \url{https://arxiv.org/abs/2012.13887}
\bibitem{MMMB} F. Merle. Determination of blow-up solutions with minimal mass for nonlinear Schr\"{o}dinger equations with critical power. \textit{Duke Math. J.} 69 (1993), no. 2, 427–454.
\bibitem{MRO} F. Merle and P. Raphael. On universality of blow-up profile for $L^2$ critical nonlinear Schr\"{o}dinger equation. \textit{Invent. Math.} 156 (2004), no. 3, 565–672.
\bibitem{MRUPB} F. Merle and P. Raphael. The blow-up dynamic and upper bound on the blow-up rate for critical nonlinear Schr\"{o}dinger equation. \textit{Ann. of Math.} (2) 161 (2005), no. 1, 157–222.
\bibitem{MRUDB} F. Merle and P. Raphael. On a sharp lower bound on the blow-up rate for the $L^2$ critical nonlinear Schr\"{o}dinger equation. \textit{J. Amer. Math. Soc.} 19 (2006), no. 1, 37–90.
\bibitem{RSEU} P. Rapha\"{e}l and J. Szeftel. Existence and uniqueness of minimal blow-up solutions to an inhomogeneous mass critical NLS. \textit{J. Amer. Math. Soc.} 24 (2011), no. 2, 471–546.
\bibitem{WL} M. Weinstein. Lyapunov stability of ground states of nonlinear dispersive evolution equations. \textit{Comm. Pure Appl. Math.} 39 (1986), no. 1, 51–67.
\bibitem{C} C. R\'{e}mi. Nonlinear Schr\"{o}dinger equations with repulsive harmonic potential and applications. \textit{SIAM J. Math. Anal.} 35 (2003), no. 4, 823--843.
\bibitem{CN} C. R\'{e}mi and N. Yoshihisa. Nonlinear Schr\"{o}dinger equations with Stark potential. \textit{Hokkaido Math. J.} 33 (2004), no. 3, 719--729.
\bibitem{WGS} M. Weinstein. Nonlinear Schr\"{o}dinger equations and sharp interpolation estimates. \textit{Comm. Math. Phys.} 87 (1982/83), no. 4, 567–576.
\end{thebibliography}
\end{document}